\documentclass{amsart}%
\usepackage{bbm}
\usepackage{amsfonts}
\usepackage{mathrsfs}
\usepackage{graphicx,verbatim}
\usepackage{amscd}
\usepackage{amsmath}
\usepackage[all,cmtip]{xy}
\usepackage{amssymb}
\usepackage{graphicx}
\usepackage{hyperref}%
\newcommand{\bbP}{\mathbb{P}}

\newcommand{\bbC}{\mathbb{C}}

\newcommand{\calO}{\mathcal{O}}

\newcommand{\calL}{\mathcal{L}}

\newcommand{\calV}{\mathcal{V}}
\newcommand{\calW}{\mathcal{W}}

\newcommand{\calR}{\mathcal{R}}

\newtheorem{theorem}{Theorem}
\newtheorem{lemma}[theorem]{Lemma}
\newtheorem{proposition}[theorem]{Proposition}
\newtheorem{definition}[theorem]{Definition}
\newtheorem{corollary}[theorem]{Corollary}

\newtheorem{remark}[theorem]{Remark}
\newtheorem{example}[theorem]{Example}

\newcommand{\bthm}{\begin{theorem}}
\newcommand{\ethm}{\end{theorem}}
\newcommand{\blem}{\begin{lemma}}
\newcommand{\elem}{\end{lemma}}
\newcommand{\bprop}{\begin{proposition}}
\newcommand{\eprop}{\end{proposition}}
\newcommand{\bdefn}{\begin{definition}}
\newcommand{\edefn}{\end{definition}}
\newcommand{\brmk}{\begin{remark}}
\newcommand{\ermk}{\end{remark}}
\newcommand{\bcor}{\begin{corollary}}
\newcommand{\ecor}{\end{corollary}}
\newcommand{\beg}{\begin{example}}
\newcommand{\eeg}{\end{example}}
\newcommand{\bitem}{\begin{itemize}}
\newcommand{\eitem}{\end{itemize}}

\usepackage{setspace}

\begin{document}

\title[Cox rings and Flag Varieties]{Cox rings of rational surfaces and \\flag varieties of $ADE$-types}
\author{Naichung Conan Leung}
\address{The Institute of Mathematical Sciences and Department of Mathematics\\
The Chinese University of Hong Kong\\
Shatin, N.T., Hong Kong.}
\email{leung@math.cuhk.edu.hk}
\author{Jiajin Zhang}
\address{The Department of Mathematics, Sichuan University, Chengdu, 610065, P.R. China}
\email{jjzhang@scu.edu.cn}

\begin{abstract}
The Cox rings of del Pezzo surfaces are closely related to the Lie groups $E_n$.
In this paper, we generalize the definition of Cox rings to $G$-surfaces defined by us earlier,
where the Lie groups $G=A_n, D_n$ or $E_n$.
We show that the Cox ring of a $G$-surface $S$ is closely related to an irreducible representation $V$ of $G$, and is generated by degree one elements.
The Proj of the Cox ring of $S$ is a sub-variety of the orbit of the highest weight vector in $V$,
and both are closed sub-varieties of $\mathbb{P}(V)$ defined by quadratic equations.
The GIT quotient of the Spec of such a Cox ring by a natural torus action is considered.

\end{abstract}
\maketitle
\sloppy









\section{Introduction}\label{introduction}

This is a continuation of our studies in which flat $G$-bundles over an
elliptic curve are related to rational surfaces $S$ of type $G$, where $G$ is
a Lie group of simply laced type in \cite{LZ1} and non-simply laced type in
\cite{LZ2}. The affine $E_{n}$ case is considered in \cite{LXZ}.
These studies generalize a classical result of Looijenga (\cite{Lo1}, \cite{Lo2}), Friedman-Morgan-Witten (\cite{FMW}), Donagi (\cite{Don})
and so on, about the case of $G=E_n$ and del Pezzo surfaces.

For instance, an $E_{n}$-surface $S$ is simply a blowup of a {\it del Pezzo} surface $X_n$ of
degree $9-n$ at a general point, where $X_n$ is a blowup of $\mathbb{P}^{2}$ at $n$ points in general position.
The del pezzo surface $X_n$ is well-known to be closely linked to
$E_{n}$ (\cite{Dem}, \cite{Ma}). For example, the orthogonal complement of the canonical class $K_{X_n}$ in $H^{2}%
(X_n,\mathbb{Z})$, equipped with the natural intersection product, is the root lattice
of $E_{n}$ (\cite{Dem}, \cite{Ma}), where  we extend the exceptional $E_{n}$-series to $0\leq n\leq8$ by setting $E_{0}%
=0,E_{1}=\mathbb{C},E_{2}=A_{1}\times\mathbb{C}$, $E_{3}=A_{2}\times
A_{1},E_{4}=A_{4}$, and $E_{5}=D_{5}$. Recall that for the del Pezzo surface $X_n$, a curve $l$ is called a
{\it line} if $l^{2}=l\cdot K_{X_n}=-1$ (which is really  of
degree $1$ under the anti-canonical morphism for $n\leq 7$). In \cite{LZ1}, we use
these root lattices and lines to construct an adjoint principal $E_{n}$-bundle
$\mathcal{E}_{n}$ over $X_{n}$ and its representation bundle (that
is, an associated principal $E_{n}$-bundle) $\mathcal{L}_{E_n}$ over
$X_n$ (corresponding to the left-end node in the Dynkin diagram, see
Figure 1).

 In Section 2.1, we describe a $D_{n}$-surface (resp. an $A_{n}$-surface) $S$ as a rational
surface with a fixed ruling $S\rightarrow\mathbb{P}^{1}$ (resp. a fixed birational
morphism $S\rightarrow\mathbb{P}^{2}$).
Note that the description of $A_n$-surfaces is slightly different from the description in \cite{LZ1}, where it is more indirect.
Here we use a more direct description to obtain the same root lattice. The results about $A_n$-surfaces cited from \cite{LZ1} are all about lattice structures and hence keep true.
Similar to the $E_{n}$-surface case, there is an adjoint principal $D_{n}$-bundle $\mathcal{D}_{n}$
(resp. an adjoint principal $A_{n}$-bundle $\mathcal{A}_{n}$) over a $D_{n}$-surface (resp.
an $A_{n}$-surface) and an associated bundle $\mathcal{L}_{D_{n}}$ (resp.
$\mathcal{L}_{A_{n}}$) determined by the lines on this surface. For simplicity, we also use $\mathcal{L}_{G}$ to denote the bundle $\mathcal{L}_{E_{n}}$,
$\mathcal{L}_{D_{n}}$ or $\mathcal{L}_{A_{n}}$, in the context.

Moreover, both the vector space $V=H^{0}(S,\mathcal{L}_{G}) $ and any fiber of the bundle $\mathcal{L}_{G}$ are representations of $G$.
The vector space $V$, or a subspace of it (denoted still by $V$), is just
the corresponding fundamental representation of $G$ determined by the
left-end node $\alpha_{L}=\alpha_{n}$. Thus  we have $G/P\subset\mathbb{P}(V)$, where $P$ is the maximal
parabolic subgroup of $G$ associated with $\alpha_{n}$.

In the classical $G=E_{n}$ case, the representations and the
flag varieties $G/P$ are related to the {\it Cox rings} of the del Pezzo surfaces $X_{n}$.

The notion of Cox rings is introduced by D. Cox in \cite{Cox} and formulated by Hu-Keel in \cite{HK}.
Let $X$ be an algebraic variety. Assume that the Picard group ${\rm Pic}(X)$ is
freely generated by the classes of divisors $D_{0},D_{1},\cdots,D_{r}$. Then
the {\it total homogeneous coordinate ring}, or the {\it Cox ring} of $X$ with respect to this basis is
given by
\[
Cox(X):=\bigoplus_{(m_{0},...,m_{r})\in\mathbb{Z}^{r+1}} H^{0}(X,\mathcal{O}%
_{X}(m_{0}D_{0} +\cdots+m_{r}D_{r})),
\]
with multiplication induced by the multiplication of functions in the function
field of $X$. Different choices of bases yield (non-canonically) isomorphic Cox rings.

The Cox ring of $X$ is naturally graded by ${\rm Pic}(X)$. Moreover, in the two-dimensional case, it is also
graded by $deg(D):=(-K_{X})D$, where $-K_{X}$ is the anti-canonical class of $X$.

In \cite{Cox}, it is shown that for a toric variety $X$, $Cox(X)$
is a polynomial ring with generators $t_{E}$, where $E$ runs over
the irreducible components of the boundary $X\setminus U$ and $U$ is
the open torus orbit. For a smooth del Pezzo surface $X_n$ of degree at most $6$, $Cox(X_n)$ is finitely generated by
sections of degree one elements (which are sections of $-1$ curves
for $n\leq7$; and in the $X_{8}$ case, sections of $-1$ curves and two
linearly independent sections of $-K_{X_8}$), and these generators
satisfy a collection of quadratic relations (see \cite{BP}, \cite{Der}, \cite{LV}, \cite{TVV} etc).
Thus in particular, a smooth del Pezzo surface is a Mori Dream Space
in the sense of Hu-Keel  (\cite{HK}), and as a result, the GIT quotient of
$Spec(Cox(X_n))$ by the action of the N\'{e}ron-Severi torus $T_{NS}$ of $X_n$ is isomorphic to
$X_n$ (\cite{HK}).

The Cox rings of del Pezzo surfaces are closely related to universal torsors and homogeneous
varieties (see for example \cite{Der}, \cite{HT}, \cite{SS}, \cite{SS2} etc).
For the Lie group $G=E_{n}$ with $4\leq n\leq8$, it is shown that there are the following two successive embeddings

\[
Proj(Cox(X_{n}))\hookrightarrow G/P\hookrightarrow\mathbb{P}(V),
\]
where $V$ is the fundamental representation associated with the left-end node
$\alpha_{L}$ in the Dynkin diagrams (see Figure 1,2,3), and the $Proj$ is considered with respect to the anti-canonical grading.

Motivated from above, we want to give a geometric description of
above results in terms of the representation bundle $\mathcal{L}_{G}$
and also generalize these results to all $ADE$ cases. In this paper, we show
how  the Lie groups, the representations and the flag varieties are
tied together with the rational surfaces.

For this, let $S$ be a {\it $G$-surface} (Definition~\ref{ADE-surface}) with $G$ a
simple Lie group of  simply laced type. Let $\mathcal{L}_G$ be the fundamental
representation bundle over $S$ determined by {\it lines}. Let
$\mathcal{W}$ be the fundamental representation bundles determined by {\it
rulings} (see Section~\ref{quadratic-form}). Let ${Sym}^2\calL_G$ be the second symmetric
power of $\calL_G$. Let $P$ be the maximal parabolic subgroup of $G$
associated with $\calL_G$.

Our main results are the following:

\begin{theorem}{\rm(Theorem~\ref{mainThm1})}
Let $S$, $G$, $\mathcal{L}_G$ and $\mathcal{W}$ be as above.
There is a canonical fiberwise quadratic form $\mathcal{Q}$ on $\mathcal{L}_G$,%
\[
\mathcal{Q}: \mathcal{L}_G\rightarrow {Sym}^2\mathcal{L}_G\rightarrow\mathcal{W}\text{,}%
\]
such that $ker(\mathcal{Q})  \subset\mathbb{P}\left(  \mathcal{L}_G%
\right)  $ is a fiber bundle over $S$ with fiber being the
homogeneous variety $G/P$, where $ker(\mathcal{Q})$ is the
subscheme of $\bbP(\calL_G)$ defined by $x\in \bbP(\calL_G)$, such that
$\mathcal{Q}(x)=0$.

Moreover, by taking global sections, we realize $G/P$ as a subvariety of
$\mathbb{P}(H^{0}(S,\mathcal{L}_G))$ cut out by quadratic equations, for $G\neq E_8$.
For $G=E_8$, we should replace $H^{0}(S,\mathcal{L}_G)$ by a subspace $V$ of dimension $248$.

\end{theorem}

We have a uniform definition for an $ADE$-surface in \cite{LZ1} (see also Section~\ref{section-surface}). Using this
definition, we can give a uniform definition of the Cox ring of a $G$-surface
$S$ (Definition~\ref{Cox-def}), where $G$ is the $ADE$ Lie group. For $G=E_{n}$,
it turns out that the Cox ring of an $E_n$-surface $S$ is the same as the Cox ring of a del Pezzo surface $X_{n}$ of degree $9-n$.
Let $Cox(S,G)$ be the Cox ring of a $G$-surface $S$. Let $T_{G}\subseteq P$ be the maximal subtorus of $G$,
and $T_{S,G}$ be the torus defined in Section~\ref{section-GIT}.

\begin{theorem}{\rm(Theorem~\ref{mainThm2},  ~\ref{mainThm3},  ~\ref{mainThm4} and Proposition~\ref{equivariant}, ~\ref{GIT-quotient})}

(1) The Cox ring of an $ADE$-surface $S$ is generated by degree $1$ elements,
and the ideal of relations between the degree $1$ generators  is generated by quadrics.

(2) We have $\mathbb{C}^*\times T_G$-equivariant embeddings:
\[
Spec(Cox(S,G))\hookrightarrow C(G/P)\hookrightarrow H^{0}(S,\mathcal{L}_G).
\]
Taking the $Proj$, we have $T_G$-equivariant embeddings:
\[
Proj(Cox(S,G))\hookrightarrow G/P\hookrightarrow\mathbb{P}(H^{0}%
(S,\mathcal{L}_G)).
\]
Both of the first two spaces are embedded into the last space as sub-varieties
defined by quadratic equations.

(3) The GIT quotient of $Spec(Cox(S,G))$ by the action of the torus $T_{S,G}$ is respectively $X_n$
for $G=E_n$, $\bbP^1$ for $G=D_n$, and a point for $G=A_n$.
\end{theorem}

Thus, we have a uniform description for Cox rings of $ADE$-surfaces and their relations to configurations of curves, representation theory and flag varieties,
as is the purpose of this paper.

Note that in the $E_6$ and $E_7$ cases, the proof of the embedding $Proj(Cox(S,G))\hookrightarrow G/P$ was achieved by Derenthal in \cite{Der} with the help of a computer program.
Trying to simplify this proof is also a very interesting question.
For $G=E_8$, the embedding was proved by Serganova and Skorobogatov (\cite{SS2}). These results about $E_n(4\leq n\leq 8)$ answer a conjecture of Batyrev and Popov (\cite{BP}).
Here we just cite their results without new proofs.

{\bf Acknowledgements}. We would like to thank the
referees very much for their very careful reading and very
instructive suggestions which make this paper much more self-contained and improve this paper greatly.
The work of the first author was partially
supported by grants from the Research Grants Council of the Hong Kong
Special Administrative Region, China (Project No. CUHK401411).


\section{$ADE$-surfaces and associated principal  $G$-bundles} \label{section-surface}

Let $G=A_n, D_n$, or $E_n$ be a complex (semi-)simple Lie groups.
In this section, we first briefly recall the definitions and constructions of $G$-surfaces and associated principal $G$-bundles from \cite{LZ1}.
After that, we study the quadratic forms defined fiberwise over these associated principal $G$-bundles.
\subsection{$ADE$-surfaces}\label{subsection-surface}

The definition of $ADE$-surfaces is motivated from the classical del Pezzo surfaces (\cite{LZ1}).
According to the results of \cite{Ma} and \cite{LZ1}, over a del Pezzo surface $X_n(0\leq n\leq 8)$ of degree $9-n$, there
is a root lattice structure of the Lie group $E_n$, and the lines and the rulings in $X_n$ can be related to
the fundamental representations associated with the endpoints of the Dynkin diagram, via a natural way.
Inspired by these, we can consider general $G$-surfaces, where $G=A_n, D_n$, or $E_n$.

When the simply laced Lie group $G$ is simple, that is, $G=E_{n}\mbox{ for } 4\leq n\leq8$, $A_{n}\mbox{ for }n\geq1$, or
$D_{n}\mbox{ for }n\geq3$, we gave a uniform definition of $ADE$-surfaces in \cite{LZ1},
using the pair $(S,C)$. It turns out that when $G=E_n$, after blowing down an exceptional curve, we obtain the classical del Pezzo surfaces $X_n$.

{\bf Notations}.
Let $h$ be the (divisor, the same below) class of a line in $\mathbb{P}^2$. Fix a rulled surface structure of $\mathbb{P}^1\times\mathbb{P}^1$ or $\mathbb{F}_1$ over $\mathbb{P}^1$, and
 let $f,s$ be the classes of a fiber and a section in the natural projection from $\mathbb{P}^1\times\mathbb{P}^1$ or $\mathbb{F}_1$ to $\mathbb{P}^1$.
 If $S$ is a blowup of one of these surfaces, then we use the same notations to denote the pullback class of $h,f,s$, and use $l_i$ to denote the exceptional class
 corresponding to the blowup at a point $x_i$. Let $K_S$ be the canonical class of $S$. Since for $S$ the Picard group and the divisor class group are isomorphic,
 we use $Pic(S)$ to denote the divisor class group of $S$.
 The Picard group ${\rm Pic}(S)$ is generated by $h,l_1,\cdots,l_n$
 or by $f,s,l_1,\cdots,l_n$ respectively.

\begin{definition}\label{ADE-surface}
Let $(S,C)$ be a pair consisting of a smooth rational
surface $S$  and a smooth rational curve $C\subset S$ with $C^{2}\neq4$. The
pair $(S,C)$ is called  an $ADE$-surface, or a $G$-surface for the Lie group $G=A_n, D_n$ or $E_n$ if it satisfies
the  following two conditions:

(i) any rational curve on $S$ has a  self-intersection number at
least $-1$;

(ii) the sub-lattice $\langle K_{S},C\rangle^{\perp}$ of ${\rm Pic}(S)$ is an
irreducible root lattice  of rank equal to $r-2$, where $r$ is the rank of
${\rm Pic}(S)$.
\end{definition}

The following proposition shows that such surfaces can be classified into
three types, and the curve $C$ in fact sits in the negative part of the Mori cone.

\begin{proposition}{\rm (\cite{LZ1}, Proposition 2.6)}\label{Classification-ADE-surfaces}
Let $(S,C)$ be an $ADE$-surface. Let
$n=rank({\rm Pic}(S))-2$. Then $C^{2}\in\{-1,0,1\}$ and

(i) when $C^{2}=-1$, $\langle K_{S},C\rangle^{\perp}$ is of $E_{n}$-type,
where $4\leq n\leq8$;

(ii) when $C^{2}=0$, $\langle K_{S},C\rangle^{\perp}$ is of $D_{n}$-type,
where $n\geq3$;

(iii) when $C^{2}=1$, $\langle K_{S},C\rangle^{\perp}$ is of $A_{n}$-type.
\end{proposition}

In the following corollary, $n$ points on $\mathbb{P}^2$ or $\mathbb{P}^1\times\mathbb{P}^1$ or $\mathbb{F}_1$ are said to be {\it in general position},
if the surface obtained by blowing up these points contains no irreducible rational curves with self-intersection number less than or equal to $-2$.

\begin{corollary}\label{surface-description}
Let $(S,C)$ be an $ADE$-surface.

(i) In the $E_{n}$ case, blowing down the $(-1)$ curve $C$ of $S$, we obtain a del Pezzo
surface $X_n$ of degree $9-n$.

(ii) In the $D_{n}$ case, $S$ is just a blowup of $\mathbb{P}^{1}\times
\mathbb{P}^{1}$ or $\mathbb{F}_{1}$ at $n$ points in general position with $C$
as the natural ruling.

(iii) In the $A_{n}$ case, the linear system $|C|$ defines a birational
map $\varphi_{_{|C|}}:S\rightarrow\mathbb{P}^{2}$. Therefore $S$ is
just the blowup of $\mathbb{P}^{2}$ at $n+1$ points in general position, and $C$ is a smooth curve which represents the class
determined by lines in $\mathbb{P}^{2}$.
\end{corollary}

\begin{corollary}\label{weight-lattice}
Let $(S,C)$ be an $ADE$-surface, and $G$ be the corresponding simple Lie group.
The lattice ${\rm Pic}(S)/(\mathbb{Z}C+\mathbb{Z}K_S)$ is the corresponding weight lattice.
Hence its dual $Hom({\rm Pic}(S)/(\mathbb{Z}C+\mathbb{Z}K_S),\mathbb{C}^*)$ is a maximal torus of $G$.
\end{corollary}
\begin{proof}
The intersection pairing $$\langle C,K_S\rangle^{\perp}\times {\rm Pic}(S)\rightarrow \mathbb{Z}$$
induces a perfect non-degenerate pairing
$$\langle C,K_S \rangle^{\perp}\times {\rm Pic}(S)/(\mathbb{Z}C+\mathbb{Z}K_S)\rightarrow \mathbb{Z}.$$
Since $\langle C,K_S\rangle^{\perp}$ is the (simply laced) root lattice of $G$, ${\rm Pic}(S)/(\mathbb{Z}C+\mathbb{Z}K_S)$
is the weight lattice of $G$. And the last statement follows since $G$ is simply connected.
\end{proof}

For convenience, we draw the Dynkin diagrams of the root lattices $\langle K_{S},C\rangle^{\perp}$ for the given $ADE$-surfaces $(S,C)$ as Figures 1-3.


\begin{equation*}  \label{figure1}
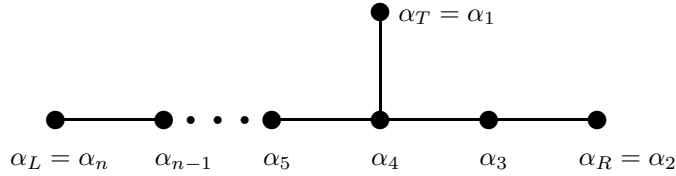
\underset{\text{Figure 1. The root system }E_n: \alpha_1=-h+l_1+l_2+l_3, \alpha_i=l_i-l_{i-1}, 2\leq i\leq n}{\setlength{%
\unitlength}{1.2cm}\begin{picture}(6, 3) \put(0, 1){\circle*{.2}} \put(1.2,
1){\circle*{.2}} \put(2.4, 1){\circle*{.2}} \put(3.6, 1){\circle*{.2}}
\put(4.8, 1){\circle*{.2}} \put(6,
1){\circle*{.2}}\put(3.6,2.2){\circle*{.2}} \put(1.5, 1){\circle*{.1}}
\put(1.8, 1){\circle*{.1}} \put(2.1, 1){\circle*{.1}} \put(0, 1){\line(1,
0){1.2}} \put(2.4, 1){\line(1, 0){1.2}} \put(3.6, 1){\line(1, 0){1.2}}
\put(4.8,1){\line(1,0){1.2}} \put(3.6, 1){\line(0, 1){1.2}} \put(-0.5,
.5){$\alpha_L=\alpha_n$} \put(1.1, .5){$\alpha_{n-1}$} \put(2.3,
.5){$\alpha_5$} \put(3.5, .5){$\alpha_4$} \put(4.7, .5){$\alpha_3$}\put(5.8,
.5){$\alpha_R=\alpha_2$} \put(3.8, 2.1){$\alpha_T=\alpha_1$} \end{picture}\
\ }
\end{equation*}

\begin{equation*}  \label{figure2}
\underset{\text{Figure 2. The root system } D_{n}: \alpha_1=-f+l_1, \alpha_i=l_i-l_{i-1}, 2\leq i\leq n}{\setlength{%
\unitlength}{1.2cm}\begin{picture}(5, 3) \put(0, 1){\circle*{.2}} \put(1.2,
1){\circle*{.2}}\put(2.4,1){\circle*{.2}}\put(3.6,1){\circle*{.2}}%
\put(4.8,1){\circle*{.2}} \put(3.6,2.2){\circle*{.2}} \put(1.5,
1){\circle*{.1}} \put(1.8, 1){\circle*{.1}} \put(2.1, 1){\circle*{.1}}
\put(0, 1){\line(1, 0){1.2}} \put(2.4, 1){\line(1, 0){1.2}} \put(3.6,
1){\line(1, 0){1.2}} \put(3.6, 1){\line(0, 1){1.2}} \put(-0.5,
.5){$\alpha_L=\alpha_n$} \put(1, .5){$\alpha_{n-1}$} \put(2.3,
.5){$\alpha_4$}\put(3.5, .5){$\alpha_3$} \put(4.5,
.5){$\alpha_R=\alpha_2$}\put(3.8, 2.1){$\alpha_T=\alpha_1$} \end{picture}\ }
\end{equation*}

\begin{equation*}  \label{figure3}
\underset{\text{Figure 3. The root system } A_{n}: \alpha_i=l_{i+1}-l_{i}, 1\leq i\leq n}{\setlength{%
\unitlength}{1.2cm}\begin{picture}(5, 2) \put(0, 1){\circle*{.2}} \put(1.2,
1){\circle*{.2}} \put(2.4, 1){\circle*{.2}} \put(3.6, 1){\circle*{.2}}
\put(4.8, 1){\circle*{.2}} \put(0,1){\line(1, 0){1.2}} \put(1.5,
1){\circle*{.1}} \put(1.8, 1){\circle*{.1}}\put(2.1, 1){\circle*{.1}}
\put(2.4, 1){\line(1, 0){1.2}} \put(3.6, 1){\line(1, 0){1.2}} \put(-0.5,
.5){$\alpha_L=\alpha_n$} \put(1, .5){$\alpha_{n-1}$} \put(2.3,
.5){$\alpha_3$} \put(3.5, .5){$\alpha_2$} \put(4.6, .5){$\alpha_R=\alpha_1$}
\end{picture}\ \ }
\end{equation*}


In these Dynkin diagrams, we specify three special nodes: the top node $\alpha_T$, the right-end node $\alpha_R$ and the left-end node $\alpha_L$, if any.
These special nodes determine three fundamental representation bundles.

\begin{definition}\label{def-line-ruling-degree}
Let $(S,C)$ be an $ADE$-surface. \par
\begin{itemize}
\item[(1)] A class $l\in {\rm Pic}(S)$ is called a line if $l^2=lK_S=-1$ and $lC=0$.\par
\item[(2)] A class $r\in {\rm Pic}(S)$ is called a ruling if $r^2=0, rK_S=-2$ and $rC=0$.\par
\item[(3)] A section $s_D\in H^0(S,\calO_S(D))$ is called of degree $d$, if $D(-K_S)=d$.
\end{itemize}
We denote the root system of the root lattice in Proposition~\ref{Classification-ADE-surfaces} (respectively, the set of lines, the set of rulings) by $R(S,C)$
(respectively, $I(S,C), J(S,C)$).
\end{definition}

Note that there is a $\mathbb{Z}$-basis for ${\rm Pic}(S)$, such that all these sets and the curve $C$ can be written down concretely (see \cite{LZ1} for details).
The adjoint principal $G$-bundle (where $G$ is of rank $n$) is
\[
\mathcal{G}:=\calO_S^{\bigoplus n}\bigoplus\bigoplus_{\alpha\in R(S,C)%
}\calO_S(\alpha).
\]

The fundamental representation bundles determined by $\alpha_L$, denoted by $\mathcal{L}_G$,
are the following (see \cite{LZ1} for details):

For $G=E_n$ with $3\leq n\leq7$,
\[
\mathcal{L}_{E_n}:=\bigoplus_{l\in I(S,C)}\calO_S(l);
\]
and for $G=E_8$,
\[
\mathcal{L}_{E_8}:=\calO_S(-K_S)^{\bigoplus8}\bigoplus\bigoplus_{l\in I(S,C)%
}\calO_S(l)\cong\mathcal{E}_{8}\otimes\calO_S(-K_S).
\]

For $G=D_n$ and $A_n$,
\[
\mathcal{L}_{G}:=\bigoplus_{l\in I(S,C)}\calO_S(l).
\]

For $G=E_n, 3\leq n\leq 7$, the rulings in corresponding surfaces are used to construct the fundamental representation bundles
$\mathcal{R}_{E_n}$ determined by $\alpha_R$ (see \cite{LZ1} for details):

For $G=E_n$ with $3\leq n\leq6$,
\[
\mathcal{R}_{E_n}:=\bigoplus_{D\in J(S,C)}\calO_S(D).
\]

For $G=E_7$,
\[
\mathcal{R}_{E_7}:=\calO_S(-K_S)^{\bigoplus7}\bigoplus\bigoplus_{D\in J(S,C)%
}\calO_S(D)\cong\mathcal{E}_{7}\otimes\calO_S(-K_S).
\]

We summarize some facts from \cite{LZ1} about these representation bundles in the following lemma.
\begin{lemma}\label{weight-set}
For any irreducible representation $V_{\lambda}$ of $G$ with the highest weight $\lambda$, denote by $\Pi(\lambda)$ or $\Pi(V_{\lambda})$ the set of all weights of $V_{\lambda}$.
\begin{itemize}
\item[(i)] For $G=A_{n-1}, D_n$, or $E_n$, the exceptional class $l_n$ represents the highest weight associated with $\alpha_L$.
Therefore $\Pi(l_n)=I(S,C)$ for $G\neq E_8$; and $\Pi(l_8)=I(S,C)\cup\{-K_S\}$ for $G=E_8$.
\item[(ii)] For $G=E_n$, the class $h-l_1$ represents the highest weight associated with $\alpha_R$.
Therefore $\Pi(h-l_1)=J(S,C)$ for $3\leq n\leq 6$; $\Pi(h-l_1)=J(S,C)\cup\{-K_S\}$ for $n=7$;
and $J(S,C)\subsetneq \Pi(h-l_1)$ for $n=8$.
\end{itemize}
\end{lemma}
\begin{proof}
(i) According to Figure 1, 2 and 3, by the definition of the pairing between weights and roots in Page 759 of \cite{LZ1}, we see that $l_n(\alpha_L)=-l_n\cdot \alpha_L=1$,
while $l_n(\alpha_i)=-l_n\cdot \alpha_i=0$, if $\alpha_i\neq\alpha_L$. Thus $l_n$ represents the highest weight associated with $\alpha_L$.

For $G\neq E_8$, $l_n$ is minuscule (that is, $W(G)$ acts on $\Pi(l_n)$ transitively), and by \cite{LZ1}, $W(G)$ acts on $I(S,C)$ transitively.
Therefore $\Pi(l_n)=I(S,C)$.

For $G=E_8$, $-K_S\in \Pi(l_8)$ because $-K_S=l_8-(-3h+l_1+\cdots+l_7+2l_8)$ and $-3h+l_1+\cdots+l_7+2l_8$ is a positive root of $E_8$.
In fact, $-K_S$ is the zero weight in $\Pi(l_8)$ (that is, $W(E_8)$ acts on $-K_S$ trivially).
Now $l_n$ is quasi-minuscule (that is, $W(G)$ acts on non-zero weights of $\Pi(l_n)$ transitively), and by \cite{Ma} or \cite{LZ1}, $W(G)$ acts on $I(S,C)$ transitively.
Therefore $\Pi(l_8)=I(S,C)\cup\{-K_S\}$.

(ii) The proof is similar.
\end{proof}


\subsection{Quadratic forms over associated bundles}\label{quadratic-form}

Let $V_{\lambda}$ be a fundamental representation of a semisimple Lie group
$G$ with the fundamental weight $\lambda$. Let ${Sym}^2V_{\lambda}$ be the second symmetric product of $V$.
 Since $2\lambda$ is the highest weight in the weight set of ${Sym}^2V_{\lambda}$, $V_{2\lambda}$  is a summand of the representation ${Sym}^2V_{\lambda}$,
 where $V_{2\lambda}$ is the fundamental representation
associated with the highest weight $2\lambda$.
Therefore there is another representation $W$ such that  ${Sym}^2V_{\lambda}=W \bigoplus V_{2\lambda}$.

With the help of the program LiE (\cite{LiE}), we list the decomposition of ${Sym}^2V_{\lambda}$ for simply laced Lie group $G$ with $\lambda$ the
fundamental weight associated with $\alpha_L=\alpha_n$ (see Figure 1, 2 and 3).

In the $G=E_n$ case, for $4\leq n \leq 6$,
$W$ is a non-trivial irreducible $G$-module of the
least dimension, which is a minuscule representation of $G$. If $r = 7$, then $W$ is the adjoint
representation, which is quasi-minuscule (that is, all the non-zero weights have multiplicity $1$
and form one orbit of the Weyl group $W(E_7)$ of $E_7$). If $r=8$, then $W=W_1\bigoplus\bbC$, where $W_1$ is
the irreducible representation associated with the node $\alpha_R$ (of dimension $3875$), and $\bbC$ is
the trivial representation.

In the $G=D_n$ case, $W=\bbC$ is the trivial representation.

In the $G=A_n$ case, $W=\{0\}$, that is, ${Sym}^2V_{\lambda}=V_{2\lambda}$.

   Let $P$ be the maximal parabolic subgroup of $G$ corresponding to the fundamental representation $V_{\lambda}$.
Then we have a homogeneous variety $G/P$. It is well-known that
$G/P\hookrightarrow \bbP(V_{\lambda})$ is a subvariety defined by quadratic relations (\cite{Li}).
A way to write explicitly the quadratic relations is the following. Let $C(G/P)$ be the affine cone over $G/P$.
Let $pr$ be the natural projection ${Sym}^2V_{\lambda} \rightarrow W$,
and $Ver : V_{\lambda} \rightarrow {Sym}^2V_{\lambda}$ be the Veronese map $x \mapsto x^2$,
then it is well known that $C(G/P)$  is the fibre $(pr\circ Ver)^{-1}(0)$
(as a scheme, see \cite{BP} Proposition 4.2 and references therein).
Thus the homogeneous variety $G/P$ is defined by the quadratic form:
 $$Q: V_{\lambda}\rightarrow {Sym}^2V_{\lambda}\rightarrow W.$$

 In fact, we can show that the quadratic form could be globally defined over fundamental representation bundles:
 $$\mathcal{Q}:\calL_G\rightarrow {Sym}^2\calL_G\rightarrow \mathcal{W},$$
 such that $G/P$ is fiberwise defined by $\mathcal{Q}$.

    Let $\calL_G$ be the fundamental representation bundle defined as in the end of Section~\ref{subsection-surface} by lines on an $ADE$-surface $S$.
    That is, $\calL_G$ corresponds to the left-end node $\alpha_L$ (or equivalently,
    associated with the fundamental weight $l_n$ corresponding to $\alpha_L$ for $G=A_{n-1}, D_n$ or $E_n$, by Lemma~\ref{weight-set}).
    For a quadratic form over a vector bundle $\calL_G$, we denote $\mathcal{Q}^{-1}(0)$ the subscheme of
 $\bbP(\calL_G)$ defined by $x\in \bbP(\calL_G)$, such that $\mathcal{Q}(x)=0$.

   By Lemma~\ref{weight-set},
   $$\calL_G=\bigoplus\limits_{\mu\in\Pi(l_n)\subseteq Pic(S)}\calO_S(\mu)^{\bigoplus k_{\mu}},$$
   where the multiplicity $k_{\mu}=8$ if $\mu=-K_S$ and $G=E_8$;
   otherwise, $k_{\mu}=1$.
      Let $\Pi({Sym}^2 \calL_G)$ be the set of weights of ${Sym}^2\calL_G$ which is saturated (see Section 13.4 of \cite{Hum}).
   Then $$\Pi({Sym}^2 \calL_G)=\{\lambda_1+\lambda_2|\lambda_1,\lambda_2\in \Pi(l_n)\}\subseteq Pic(S),$$
   and
   $${Sym}^2\calL_G=\bigoplus\limits_{\mu\in\Pi({Sym}^2 \calL_G)\subseteq Pic(S)}\calO_S(\mu)^{\bigoplus m_{\mu}},$$
   where $m_\mu$ is the multiplicity uniquely determined by $\mu$ and ${Sym}^2 \calL_G$. Since $2l_n$ occurs with multiplicity one, by the saturatedness,
   $\Pi(2l_n)\subseteq \Pi({Sym}^2 \calL_G)$. Therefore ${Sym}^2 \calL_G $ contains a summand $\mathcal{V}_{2l_n}$ which is an irreducible representation bundle associated with the highest weight $2l_n$.
   We write $\mathcal{V}_{2l_n}$ as
   $$\mathcal{V}_{2l_n}=\bigoplus\limits_{\mu\in\Pi({Sym}^2 \calL_G)\subseteq Pic(S)}\calO_S(\mu)^{\bigoplus n_{\mu}},$$
   where $n_{\mu}=0$ if $\mu\notin \Pi(2l_n)$ and $1\leq n_{\mu}\leq m_{\mu}$ if $\mu\in \Pi(2l_n)$.

   The other summand $\calW$ of ${Sym}^2 \calL_G$ is automatically a representation bundle:
   $$\mathcal{W}=\bigoplus\limits_{\mu\in\Pi({Sym}^2 \calL_G)\subseteq Pic(S)}\calO_S(\mu)^{\bigoplus (m_{\mu}-n_{\mu})}.$$

   We are mainly interested in the representation bundle $\mathcal{W}$, which we discuss case by case according to $G=E_n$, $D_n$ or $A_{n-1}$.

      (i) For $G=E_n$, $h-l_1\in \Pi({Sym}^2 \calL_G)$. By \cite{LiE}, $\mathcal{W}$ contains a weight space with the weight $h-l_1$.
   Thus the set $J(S,C)$ of rullings on $S$ are contained in the set $\Pi(\mathcal{W})$ of weights of $\mathcal{W}$.
   Therefore, as a vector bundle, $\mathcal{W}$ contains $\bigoplus\limits_{\mu\in J(S,C)}\calO_S(\mu)$ as summands.
   By counting the rank of $\mathcal{W}$ (\cite{LiE}) and the number of the elements of $J(S,C)$, we find that for $4\leq n\leq 7$,
    $$\mathcal{W}=\bigoplus\limits_{\mu\in J(S,C)}\calO_S(\mu)=\calR_{E_n}$$
    is the irreducible representation bundle associated with $\alpha_R$ (Lemma~\ref{weight-set}).

        Similarly, for $G=E_8$, by \cite{LiE}, $\mathcal{W}$ is a direct sum of $\calR_{E_8}$ and a line bundle which is a trivial representation.
    Note that among the weights of ${Sym}^2\calL_G$, only $-2K_S$ appears as a zero weight (Lemma~\ref{weight-set}).
    Thus the line bundle considered here is nothing but $\calO_S(-2K_S)$.
    Therefore  $$\mathcal{W}=\calR_{E_n}\bigoplus \calO_S(-2K_S).$$

      (ii) For $G=D_n$, $f\in \Pi({Sym}^2 \calL_G)$. By \cite{LiE}, $\mathcal{W}$ is a line bundle which is a trivial representation bundle.
      Note that the only zero weight of ${Sym}^2 \calL_G$ is $f$. Therefore $\mathcal{W}\cong\calO_S(f)$.\par

      (iii) For $G=A_{n-1}$, by a dimension counting, ${Sym}^2 \calL_G\cong \calV_{2l_n}$. Therefore $\calW=0$.\par

    Thus we achieved the first statement of the following theorem.
 \begin{theorem}\label{mainThm1} The notations are as above.

  (1) We have a decomposition of representation bundles:
 $${Sym}^2\calL_G=\mathcal{W}\bigoplus\mathcal{V}_{2l_n}.$$
 Here $\mathcal{W}=\calR_{E_n}$ for $G=E_n$ with $4\leq n\leq 7$;
 $\mathcal{W}=\calR_{E_8}\bigoplus \calO_S(-2K_S)$ for $G=E_8$;
 $\mathcal{W}=\calO_S(f)$ for $G=D_n$; and $\mathcal{W}=0$ for $G=A_{n-1}$.

 (2) The projection to the first summand defines a quadratic form on $\calL_G$
 $$\mathcal{Q}:\calL_G\rightarrow {Sym}^2\calL_G\rightarrow \mathcal{W},$$ such that the homogeneous variety $G/P$ is the fiber of the subscheme
 (considered as a scheme defined over $S$) $\mathbb{P}(\mathcal{Q}^{-1}(0))\subseteq \mathbb{P}(\calL_G)$.

 (3) By taking global sections, for $G\neq E_8$, we realize $G/P$ as a subvariety of $\bbP(H^0(S,\calL_G))$,
 cut out by quadratic equations. For $G=E_8$, we replace $H^0(S,\calL_G)$ by a subspace $V$ of dimension $248$,
 where $V=\mathbb{C}\langle s_K\rangle^{\bigoplus 8}\bigoplus\bigoplus\limits_{\mu\in I(S,G)} H^0(S,\calO_S(\mu))$
 with $s_K$ a fixed non-zero global section of $\calO_S(-K_S)$.
 \end{theorem}

\begin{proof}

It remains to verify (2) and (3), which are essentially consequences of (1).

(2). Note that fiberwise, the map
$$\mathcal{Q}: \calL_G\rightarrow {Sym}^2\calL_G=\mathcal{W}\bigoplus\mathcal{V}_{2l_n}\rightarrow \calW$$
is exactly the map (Lemma~\ref{weight-set})
$$Q: V_{l_n}\rightarrow {Sym}^2 V_{l_n}\cong W\bigoplus V_{2l_n}\rightarrow W,$$
where $V_{l_n}, W$ and $V_{2l_n}$ are as in the beginning of Section~\ref{quadratic-form}.

By \cite{Li}, $Q^{-1}(0)$ is the cone over $G/P$ in $V_{l_n}$, that is $\mathbb{P}(Q^{-1}(0))=G/P\subseteq \mathbb{P}(V_{l_n})$. \par

(3). First by \cite{LZ1}, every element $\mu\in I(S,G)$ is represented by a unique irreducible curve in an $ADE$-surface $S$ and hence $\dim H^0(S,\calO_S(\mu))=1$.
For $G\neq E_8$, recall that $\calL_G=\bigoplus\limits_{\mu\in I(S,G)}\calO_S(\mu)$.
Therefore we can choose a unique global section for each summand of $\calL_G$ up to a constant.

By \cite{Li}, $C(G/P)\subseteq V_{l_n}$ is defined by finitely many quadratic polynomials. Let $f(x_{\mu}|_{\mu\in I(S,G)})$'s be such polynomials.
Let $s_{\mu}$ be the global section of $\calO_S(\mu), \mu\in I(S,G)$. Then $H^0(S,\calL_G)=\{\sum\limits_{\mu\in I(S,G)}x_{\mu}s_{\mu}|x_{\mu}\in\mathbb{C}\}$,
and the same polynomials $f(x_{\mu}|_{\mu\in I(S,G)})$'s define $G/P$.

   For $G=E_8$, since $H^0(S,\calO_S(-K_S))$ is of dimension two, we should fix any one non-zero global section $s_K$ of  $\calO_S(-K_S)$.
   Similarly by \cite{Li}, $C(G/P)\subseteq V_{l_8}$ is defined by finitely many quadratic polynomials.
   Let $f(x_{\mu}|_{\mu\in \Pi(l_8)})$'s be such polynomials.
   Thus, we take a subspace of $H^0(S,\calL_G)$ of dimension 248 as follows: $V=\mathbb{C}\langle s_K\rangle^{\bigoplus 8}\bigoplus\bigoplus\limits_{\mu\in I(S,G)} H^0(S,\calO_S(\mu))$.
   As a vector space $V=\{x_1s_{K,1}+\cdots+x_8s_{K,8}+\sum x_{\mu}s_{\mu}| x_i,x_{\mu}\in\mathbb{C}\}$ where $s_{K,i}=s_K$ is the basis of the $i$-th $\mathbb{C}\langle s_K\rangle$,
   and the same polynomials $f$'s define $G/P$.

\end{proof}


\begin{remark}
\rm{
 The bundle $\mathcal{W}$ appearing in Theorem~\ref{mainThm1} can be called the {\it representation bundle determined by rulings},
since in the $G=D_n$ and $E_n$ cases, it is constructed by using the rulings.

}
\end{remark}

\section{Cox rings of $ADE$-surfaces and flag varieties}

\subsection{Cox rings of $ADE$-surfaces}

The notion of Cox rings is defined by Cox (\cite{Cox}) for toric varieties and
he shows that for a toric variety, its Cox ring is precisely its total
coordinate ring. Hu and Keel (\cite{HK}) give a general definition of Cox
rings for $\mathbb{Q}$-factorial projective varieties $X$ with ${\rm Pic}(X)_{\mathbb{Q}}\cong N^1(X)$,
and show that it is related to GIT and Mori Dream Spaces.
Batyrev and Popov (\cite{BP}), followed by Derenthal and so on (\cite{Der}), make a
deep study on the Cox rings of del Pezzo surfaces. From their studies, for the del Pezzo surface $X_{n}$
($n\leq7$), its Cox ring is closely linked to the fundamental
representation $V$ with the highest weight $\alpha_{L}$ in the Dynkin diagram of
$E_{n}$. More precisely, the projective variety defined by this ring is
embedded into the flag variety $G/P$, where $G$ is the complex Lie group
$E_{n}$, and $P$ is the maximal parabolic subgroup determined by the node
$\alpha_{L}$. Both  $G/P$ and this variety  are subvarieties of $\mathbb{P}(V)$ defined by quadrics.

Motivated from these, we can define (generalized) Cox rings for $G$-surfaces as follows.

\begin{definition} \label{Cox-def}
Let $(S,C)$ be a $G$-surface with $G=A_n, D_n$ or $E_n$, and a $\mathbb{Z}$-basis of ${\rm Pic}(S)$ be chosen as in Section~\ref{subsection-surface}.
Then we define the Cox ring of $(S,C)$ as
\[
Cox(S,G):=\bigoplus_{D\in {\rm Pic}(S),DC=0}H^{0}(S,\calO_S(D)),
\]
with a well-defined multiplication (see Section~\ref{introduction}).
\end{definition}

Notice that $Cox(S,G)$ is naturally graded by the degree defined in Definition~\ref{def-line-ruling-degree}.

\begin{remark}
 As usual, let $X_n$ be a del Pezzo surface of degree $9-n$
with $4\leq n\leq8$, let $S\rightarrow X_n$ be a blowup at a general
point, and $C$ be the corresponding exceptional curve. Then for the $E_{n}$-surface $(S,C)$, we have
\[
Cox(S,E_n)\cong\bigoplus_{D\in {\rm Pic}(X_{n})}H^0(X_n,\calO_{X_n}(D))=
Cox(X_n).
\]
Thus the definition of Cox rings of $E_n$-surfaces is the same as the classical definition of Cox rings for del Pezzo surfaces $X_n$.
The reason for the displayed isomorphism is that the contraction morphism $\pi:S\rightarrow X_n$ induces an isomorphism
$\pi^*:{\rm Pic}(X_n)\rightarrow C^{\perp}\subseteq {\rm Pic}(S)$ such that the pull-back of rational functions
$H^0(X_n,\mathcal{O}_{X_n}(D))\rightarrow H^0(S,\mathcal{O}_S(\pi^*D))$ is an isomorphism for any divisor $D$ of $X_n$.

\end{remark}

\begin{corollary}
1) For the $D_{n}$-surface $(S,C)$, $C\equiv f$ is a smooth
fiber. Then we have
\[
Cox(S, D_{n})=\bigoplus_{D\in {\rm Pic}(S),Df=0}H^0(S,\calO_S(D)).
\]

2) For the $A_{n}$-surface $(S,C)$, $C\equiv h$ (linear equivalence) is a twisted
cubic. Then we have
\[
Cox(S, A_{n})=\bigoplus_{D\in {\rm Pic}(S),Dh=0}H^0(S,\calO_S(D)).
\]

\end{corollary}

\begin{theorem}\label{mainThm2}
\label{thm-generator}  Let $G=A_n, D_n (n\geq 3)$ or $E_n (4\leq n\leq 8)$.
The Cox ring $Cox(S,G)$ is
finitely generated, and generated by degree $1$ elements. For $G\neq E_8$,
the generators of $Cox(S,G)$ are global sections of invertible sheaves defined by
lines on $S$.  For $G=E_{8}$, we should add to the above set of generators
two linearly independent global sections of the anti-canonical sheaf on $X_{8}$.
\end{theorem}

\begin{proof}
Let $f,s,h,l_i$'s be as in Section 2.1 (Notations).

1) For the $G=E_{n}$ case, see \cite{BP}, \cite{Der}, \cite{LV} and \cite{SS}.

2) For the $G=D_{n}$ case, let $D\in {\rm Pic}(S)$ and $DF=0$.
Assume that $D$ is effective. Then we can write $D\equiv \sum a_{i}D_{i}$ (here '$\equiv$' means the linear equivalence) with $D_{i}$
irreducible curves and $a_{i}\geq 0$. Choose a smooth fiber $F$. Then
$D_{i}F\geq0$. Thus $DF=0$ implies $a_{i}=0$ or $D_{i}F=0$ for all $i$.
By the Hodge index theorem, $D_{i}F=0$ implies that $D_{i}\equiv F$ or $D_{i}=l_{j}$ or $D_{i}=f-l_{k}$, for some
$j,k$. Thus $D\equiv a_{0}F+\sum_{i} a_{i}l_{i}+\sum_{j} b_{j}(f-l_{j})$ with $a_0, a_{i}, b_{j}\geq0$.
Moreover, we can assume that $\{i\ |\ a_{i}\neq0\}\cap\{j\ |\ b_{j}\neq0\}=\emptyset$.

Let $x_{i}$ (resp. $y_{i}$) be a nonzero global section of $\calO_S(l_{i})$ (resp.
$\calO_S(f-l_{i})$).

Thus, by induction, we can show that
$$\mbox{dim}H^0(S,\mathcal{O}_S(D))=\mbox{dim}H^{0}(S, \mathcal{O}_S(a_{0}F))=a_{0}+1.$$ The proof goes as follows.
We have a short exact sequence
$$0\rightarrow \mathcal{O}_S(a_0F+(a_i-1)l_i)\rightarrow \mathcal{O}_S(a_0F+a_il_i)\rightarrow \mathcal{O}_{l_i}(a_0F+a_il_i)\rightarrow 0.$$
Note that $\mathcal{O}_{l_i}(a_0F+a_il_i)\cong \mathcal{O}_{\mathbb{P}^1}(-a_i)$, since $l_i\cong\mathbb{P}^1$ and $l_i(a_0F+a_il_i)=-a_i$.
Thus we have a long exact sequence
\begin{eqnarray*}
0 &\rightarrow& H^0(S, \mathcal{O}_S(a_0F+(a_i-1)l_i))\rightarrow H^0(S, \mathcal{O}_S(a_0F+a_il_i))\rightarrow H^0(\mathbb{P}^1, \mathcal{O}_{\mathbb{P}^1}(-a_i))\\
&\rightarrow& H^1(S, \mathcal{O}_S(a_0F+(a_i-1)l_i))\rightarrow\cdots.
\end{eqnarray*}

When $a_i\geq 1$, $H^0(\mathbb{P}^1, \mathcal{O}_{\mathbb{P}^1}(-a_i))=0$, and therefore
$$H^0(S, \mathcal{O}_S(a_0F+(a_i-1)l_i))\cong H^0(S, \mathcal{O}_S(a_0F+a_il_i)).$$
Hence by induction we have $$H^0(S, \mathcal{O}_S(a_0F))\cong H^0(S, \mathcal{O}_S(a_0F+a_il_i)).$$
By repeating this process, we have $$H^0(S, \mathcal{O}_S(D))\cong H^0(S, \mathcal{O}_S(a_0F)).$$
It remains to prove $\mbox{dim}H^{0}(S, \mathcal{O}_S(a_{0}F)))=a_{0}+1$. Also this comes from the following short exact sequence:
$$0\rightarrow \mathcal{O}_S((a_0-1)F)\rightarrow \mathcal{O}_S(a_0F)\rightarrow \mathcal{O}_{F}(a_0F)\rightarrow 0.$$
Here $\mathcal{O}_{F}(a_0F)\cong \mathcal{O}_{\mathbb{P}^1}$, since $F\cong\mathbb{P}^1$ and $(a_0F)F=0$.
Taking the long exact sequence, we have
\begin{eqnarray*}
0 &\rightarrow& H^0(S, \mathcal{O}_S((a_0-1)F))\rightarrow H^0(S, \mathcal{O}_S(a_0F))\rightarrow H^0(\mathbb{P}^1, \mathcal{O}_{\mathbb{P}^1})\\
&\rightarrow& H^1(S, \mathcal{O}_S((a_0-1)F))\rightarrow H^1(S, \mathcal{O}_S(a_0F))\rightarrow H^1(\mathbb{P}^1, \mathcal{O}_{\mathbb{P}^1})\rightarrow\cdots.
\end{eqnarray*}

Since $H^1(\mathbb{P}^1, \mathcal{O}_{\mathbb{P}^1})=0$, we shall have $H^1(S, \mathcal{O}_S(a_0F))=0$ if $H^1(S, \mathcal{O}_S((a_0-1)F))=0$.
For $a_0=1$, we have $H^1(S, \mathcal{O}_S((a_0-1)F))=H^1(S, \mathcal{O}_S)=0$, since $S$ is a rational surface. Thus by induction,
we have for all $a_0\geq 0$, $H^1(S, \mathcal{O}_S(a_0F))=0$. Then from the last long exact sequence we have
\begin{eqnarray*}
\mbox{dim} H^0(S, \mathcal{O}_S(a_0F))&=&\mbox{dim} H^0(S, \mathcal{O}_S((a_0-1)F))+\mbox{dim} H^0(\mathbb{P}^1, \mathcal{O}_{\mathbb{P}^1})\\
&=&\mbox{dim} H^0(S, \mathcal{O}_S((a_0-1)F))+1.
\end{eqnarray*}
Therefore by induction, we have $$\mbox{dim} H^0(S, \mathcal{O}_S(a_0F))=a_0+1.$$

Let $H^0(S,\mathcal{O}_S(F))=\mathbb{C}\langle v_{1},v_{2}\rangle$, where $v_1,v_2$ are two linearly independent global sections of $\mathcal{O}_S(F)$.
Then the linearly independent generators of $H^0(S,\mathcal{O}_S(D))$ can be taken as
$u_{k}(\Pi_{i} x_{i}^{a_{i}})(\Pi_{j}
y_{j}^{b_{j}})$, where $u_{k}=v_{1}^{k}v_{2}^{a_{0}-k},k=0,\cdots,a_{0}$,

Let $n\geq2$. Thus we have at least two different singular fibers:
$l_{1}+(f-l_{1})$ and $l_{2}+(f-l_{2})$. Then $x_{1}y_{1}$ and $x_{2}y_{2}$
are linearly independent elements in $H^0(S,\calO_S(F))$. Thus we can take
$v_{1}=x_{1}y_{1},v_{2}=x_{2}y_{2}$.

Therefore, the Cox ring is generated by global sections of the invertible
sheaves defined by lines (when $n\geq2$).

In fact, if $(x)=F'$ is a smooth fiber, then we must have
\[
x=a(x_{1}y_{1})+b(x_{2}y_{2}),
\]
with $a\neq0$ and $b\neq0$.

3) For the $G=A_{n}$ case, let $D\in {\rm Pic}(S)$, such that $Dh=0$. Then
obviously, $D\equiv a_{1}l_{1}+\cdots+a_{n+1}l_{n+1}$. $D\geq0$ if and only if
$a_{i}\geq0$. Let $x_i\neq 0$ be a global section of $\calO_S(l_i), 1\leq i\leq n+1$.
Note that
\[
\mbox{dim}H^0(S,\calO_S(a_{1}l_{1}+\cdots+a_{n+1}l_{n+1}))=1,
\]
and $x_{1}^{a_{1}}\cdots x_{n+1}^{a_{n+1}}$ generates the one-dimensional vector space $H^{0}%
(S, \calO_S(a_{1}l_{1}+\cdots+a_{n+1}l_{n+1}))$.
By Definition~\ref{def-line-ruling-degree},
$${\rm deg}(x_1^{a_1}\cdots x_{n+1}^{a_{n+1}}):=D(-K_S)=a_0+\cdots+a_{n+1}.$$
Thus, the Cox ring is in fact a polynomial ring with $n+1$ variables:
\[
Cox(S, A_{n})=k[x_{1},\cdots,x_{n+1}].
\]

\end{proof}

By this theorem, the Cox ring $Cox(S,G)$ of a $G$-surface $S$ is
a quotient of the polynomial ring $P(S,G)=k[x_{1},\cdots,x_{N_{G}}]$ by an
ideal $\mathcal{I}(S,G)$:
\[
Cox(S,G)=k[x_{1},\cdots,x_{N_{G}}]/\mathcal{I}(S,G),
\]
where $N_{G}$ is the number of lines (Definition~\ref{def-line-ruling-degree}) in the $G$-surface $S$ for $G\neq E_8$;
for $G= E_8$, $N_{G}$ is the number of lines plus 8.

\begin{theorem}\label{mainThm3}
For any $ADE$-surface $S$, the ideal $\mathcal{I}(S,G)$ is
generated by quadrics.
\end{theorem}

\begin{proof}
1) For $G=A_{n}$, the ideal $\mathcal{I}(S,G)=0$.

2) For $G=E_{n}$, see \cite{Der} for $4\leq n\leq7$ and \cite{SS2}, \cite{TVV} for $n=8$.

3) For $G=D_{n}$, let $x_i,y_i$ be as in the proof of Theorem~\ref{mainThm2}.
We want to show that
\[
Cox(S,D_n)=k[x_{1},y_{1},\cdots,x_{n},y_{n}]/\mathcal{I}(S,D_{n}),
\]
where
\[
\mathcal{I}(S,D_{n})=(a_{31}x_{1}y_{1}+a_{32}x_{2}y_{2}+a_{33}x_{3}%
y_{3},\cdots,a_{n1}x_{1}y_{1}+a_{n2}x_{2}y_{2}+a_{n3}x_{n}y_{n}),
\]
and all $a_{ij}\neq0$.

By the proof of Theorem~\ref{thm-generator}, we see that all the generating
relations come from the ruling $f$. The vector space
$H^{0}(S,\calO_S(f))$ is a two-dimensional space. Moreover, any two
singular fibers are different. Therefore, when $n\geq3$, any two elements of $\{x_{1}%
y_{1},\cdots,x_{n}y_{n}\}$ are linearly independent, and any three elements are
linearly dependent. Thus, the ideal $\mathcal{I}(S,D_{n})$ is of desired form.
\end{proof}

\subsection{Cox rings and flag varieties}

Let $G$ be a complex simple Lie groups, and $\lambda$ be a
fundamental weight. Let $P$ be the corresponding maximal parabolic
subgroup and $V_{\lambda}$ be the highest weight module. It is well known that
the homogeneous space $G/P$ (the orbit of the highest weight vector of
$V_{\lambda}$) could be embedded into the projective space
$\mathbb{P}(V_{\lambda})$ with quadratic relations as generating relations.
 It is  showed in \cite{Der} that
for the del Pezzo surfaces $X_n$ with $n=6,7$,
\[
Spec(Cox(X_n))\hookrightarrow C(E_{n}/P).
\]
 Here $P$ is the maximal parabolic
subgroup determined by the left-end node $\alpha_{L}$ in the Dynkin diagram
(Figure 1), and $C(E_{n}/P)$ is the affine cone over the homogeneous space
$E_{n}/P$.

  Given an $ADE$-surface $S$, we let $\mathcal{L}_{G}$ be the representation bundle determined by lines on $S$.
  The vector space of global sections
$H^0(S,\mathcal{L}_{G})$ is the  fundamental representation of $G$ associated
with the node $\alpha_{L}$ (for $G=E_{8}$ we should replace it by a subspace $V$).

The following result relates the Cox ring of a $G$-surface with the homogeneous
variety $G/P$. Thus we obtain a uniform description of
the relation between the Cox rings $Cox(S,G)$,  the homogeneous space $G/P$,
and fundamental representation bundles defined by lines in $S$,  for any Lie
group $G=A_n, D_n, E_n$.

\begin{theorem}\label{mainThm4}
Let $G=A_n, D_n$ or $E_n$. Let $S$, $\mathcal{L}_{G}$ and $P$ be as above.
 We have an embedding: $Proj(Cox(S,G))\hookrightarrow G/P$.
\end{theorem}

\begin{proof}
For $G=E_{n}$ and $4\leq n\leq7$, the result is known, see \cite{BP} and \cite{Der}.
For $G=E_8$, see \cite{SS2}.

For $G=A_{n}$, in fact we have an isomorphism: $Proj(Cox(S,G))\cong A_{n}/P
\cong \mathbb{P}^{n}$. For $G=D_{n}$, by the proof of Theorem~\ref{mainThm3},
\[
Cox(S,D_{n})=k[x_{1},y_{1},\cdots,x_{n},y_{n}]/\mathcal{I}(S,D_{n}),
\]
where
\[
\mathcal{I}(S,D_{n})=(a_{31}x_{1}y_{1}+a_{32}x_{2}y_{2}+a_{33}x_{3}%
y_{3},\cdots,a_{n1}x_{1}y_{1}+a_{n2}x_{2}y_{2}+a_{n3}x_{n}y_{n}),
\]
and all $a_{ij}\neq0$. By Lemma~\ref{quadric-equation}, the affine coordinate ring of $C(D_{n}/P)$ is
\[
k[x_{1},y_{1},\cdots,x_{n},y_{n}]/(x_{1}y_{1}+\cdots+x_{n}y_{n}).
\]

There exist nonzero $b_{i},3\leq i\leq n$, such that the coefficient of the
term $x_{i}y_{i}$ in the sum $\sum_{3\leq i\leq n}b_{i}(a_{i1}x_{1}%
y_{1}+a_{i2}x_{2}y_{2}+a_{i3}x_{i}y_{i})$ is nonzero, by dimension counting.

Therefore, we have a surjective homomorphism from the affine coordinate ring
of $C(G/P)$ to that of $Spec(Cox(S,G))$, which defines a closed embedding
\[
Spec(Cox(S,G))\hookrightarrow C(G/P).
\]

\end{proof}

The following result is well-known. But since we can not find an appropriate reference, we include its proof here.

\begin{lemma}\label{quadric-equation}
Let $V_{l_n}=\bigoplus_{\mu\in I(S,G)}V_{(\mu)}$ be the irreducible representation of $D_n=SO(2n,\mathbb{C})$ associated with the highest weight $l_n$.
Let $P$ be the maximal hyperbolic subgroup of $D_n$ associated with the highest weight $l_n$.\par

Then there exists a basis $\{u_i,v_i|i=1,\cdots,n\}$ for $V_{l_n}$ such that $D_n/P\subseteq \mathbb{P}(V_{l_n})$ is defined by the quadratic equation $Q(x_1,y_1,\cdots,x_n,y_n)=\sum_{i=1}^{n} x_iy_i=0$.
\end{lemma}
\begin{proof}
By \cite{LZ1}, $\Pi(l_n)=I(S,G)=\{l_i,f-l_i|i=1,\cdots,n\}$.
Hence we can take a basis for $V_{l_n}$ as $\{u_{i}, v_{i}|i=1,\cdots,n\}$, where for $1\leq i\leq n$, $u_{i}$ (resp. $v_{i}$) is the basis for the one-dimensional weight space with the weight $l_i$ (resp. $f-l_i$).

As in the beginning of Section~\ref{quadratic-form}, let $Q$ be the composite map $$Q: V_{l_n}\rightarrow {Sym}^2 V_{l_n}=V_{2l_n}\bigoplus W\rightarrow W,$$
where $W\cong\mathbb{C}$ is the trivial representation. By \cite{Li}, $D_n/P\subseteq \mathbb{P}(V_{l_n})$ is defined by the equation $Q=0$.
We only need to write down $Q$ explicitly.

   The following map
   \begin{eqnarray*}
     Q': V_{l_n}=\mathbb{C}^{2n}\langle u_{i}, v_{i}|1\leq i\leq n\rangle &\rightarrow& \mathbb{C}\\
         \sum (x_iu_{i}+y_iv_{i})&\mapsto& \sum x_i y_i
   \end{eqnarray*}
 defines a non-degenerate symmetric quadratic form which is $D_n$-invariant.
 (In fact, $D_n=SO(2n,\mathbb{C})$ is the Lie group preserving this non-degenerate symmetric quadratic form with determinant one.)
 Therefore by the Complete Reducibility Theorem for semi-simple Lie groups, $\mathbb{C}$ is a summand of the $D_n$-module ${Sym}^2 V_{l_n}$.
 Hence $W\cong \mathbb{C}$ and $Q=Q'$.
\end{proof}

\subsection{Cox rings and the GIT quotients}\label{section-GIT}

Let $(S,C)$ be a $G$-surface. The subset $C^{\perp}$ of ${\rm Pic}(S)$ is
a free abelian group of rank equal to $rank({\rm Pic}(S))-1$.
   We have the following short exact sequence:
$$0\rightarrow \mathbb{Z}C\rightarrow {\rm Pic}(S)\rightarrow {\rm Pic}(S)/\mathbb{Z}C\rightarrow 0.$$
   Taking the dual, we have
$$1\rightarrow Hom({\rm Pic}(S)/\mathbb{Z}C,\mathbb{C}^*)\rightarrow T_{NS}\rightarrow \mathbb{C}^*\rightarrow 1.$$
We denote the torus $Hom({\rm Pic}(S)/\mathbb{Z}C,\mathbb{C}^*)$ by $T_{S,G}$.
Note that, for $G=E_n$, $T_{S,G}$ is exactly the N\'{e}ron-Severi torus $T_{NS}$ of the del Pezzo surface $X_n$ obtained by blowing down $C$ from $S$.

   The torus $T_{S,G}$ is an extension of $\mathbb{C}^*$ by a maximal torus $T_G$ of $G$.
   One can see that the lattice $\langle C,K_S\rangle^{\perp}$ is a sublattice of $C^{\perp}$
   of rank equal to $rank(C^{\perp})-1$. In fact we have a short exact sequence:
$$0\rightarrow \mathbb{Z}K_S\rightarrow {\rm Pic}(S)/\mathbb{Z}C \rightarrow {\rm Pic}(S)/(\mathbb{Z}C+\mathbb{Z}K_S)\rightarrow 0.$$
Since the character group $\chi(T_G)$ of $T_G$ is isomorphic to the weight lattice ${\rm Pic}(S)/(\mathbb{Z}C+\mathbb{Z}K_S)$
(if we take $G$ to be the simply connected one),
we have $T_G\cong Hom({\rm Pic}(S)/(\mathbb{Z}C+\mathbb{Z}K_S),\mathbb{C}^*)$, by Corollary~\ref{weight-lattice}.
Therefore the following sequence is exact:
$$1\rightarrow T_G\rightarrow T_{S,G}\rightarrow \mathbb{C}^*\rightarrow 1.$$

The torus $T_{S,G}$ acts on $Cox(S,G)$ (and therefore acts on $Spec(Cox(S,G))$) naturally.

\begin{proposition}\label{equivariant}
The embeddings $$Proj(Cox(S,G))\hookrightarrow G/P\hookrightarrow \mathbb{P}(V_{l_n})$$
arising in Theorem~\ref{mainThm4} are $T_{G}$-equivariant.
\end{proposition}
\begin{proof}
This is known for $G=E_n$ by \cite{BP}, \cite{Der} and \cite{SS}.\par

For $G=D_n$, since our coordinate system $\{x_i,y_i|i=1,\cdots,n\}$ is chosen by the weight vectors (see Lemma~\ref{quadric-equation}), $T_G$ acts on these spaces as scalars on each coordinate.
According to Lemma~\ref{quadric-equation} and the proof of Theorem~\ref{mainThm4}, these embeddings are $T_G$-equivariant.

For $G=A_{n-1}$, it is trivial, since $Proj(Cox(S,G))\cong G/P\cong \mathbb{P}(V_{l_n})$ and $T_G$ acts on these spaces as scalars on each coordinate.

\end{proof}

In the $E_n$ case, by Hu-Keel (\cite{HK}), the GIT quotient of $Spec(Cox(X_n))$ by $T_{NS}$ is exactly the surface
$X_n$. In general, we have

\begin{proposition}\label{GIT-quotient}
Let $X_n$ be the del Pezzo surface obtained from an $E_n$-surface $S$ by blowing down $C$.
The GIT quotient of $Spec(Cox(S,G))$ by the action of $T_{S,G}$ is respectively $X_n$ for $G=E_n$,
$\mathbb{P}^1$ for $G=D_n$,
and a point for $G=A_n$.
\end{proposition}
\begin{proof}
For the case $G=E_n$, we can apply the result of Hu-Keel (Proposition 2.9 in \cite{HK}),
since $Cox(S,E_n)\cong Cox(X_n)$ is finitely generated by \cite{BP}, and since $T_{NS}(X_n)\cong T_{S,G}$.

It remains to prove the cases $G=D_n$ and $G=A_n$.

For $G=D_n$, we apply a linearization argument as in Hu-Keel (\cite{HK}).
In this case ${\rm Pic}(S)=\mathbb{Z}\langle f,s,l_1,\cdots,l_n\rangle$ and $C\equiv f$. Let $R=Cox(S,D_n)$.
Note that $R$ is naturally graded by the lattice ${\rm Pic}(S)/\mathbb{Z}f$.
For example, $a_0f+\sum_{i=1}^{n}a_i l_i\in f^{\perp}$ with $a_i\in \mathbb{Z}$ is graded by
$a_0s+ \sum_{i=1}^{n}a_i l_i\in {\rm Pic}(S)/\mathbb{Z}f$.
Then $R=\bigoplus_{v\in {\rm Pic}(S)/\mathbb{Z}f}R_v$. According to Theorem~\ref{thm-generator}, $R$ is finitely generated.
Note that $T_{S,G}$ acts naturally on $R$. So $R=\bigoplus_{v\in {\rm Pic}(S)/\mathbb{Z}f=\chi(T_{S,G})}R_v$
is the eigenspace decomposition for this action. Thus
$$H^0(Spec(Cox(S,G)), L_v)^{T_{S,G}}=R_v,$$ where $L_v$ is the line bundle determined by the
linearization $v\in {\rm Pic}(S)$. And the ring of invariants is
$$R(Spec(Cox(S,G)),L_v)^{T_{S,G}}=R(S,\calO_S(v')),$$
where $v'\in f^{\perp}$ is graded by $v$, and $R(S,\calO_S(v'))$ (similar for $R(Spec(Cox(S,G)),L_v)$) denotes the graded ring $\bigoplus_{n\geq 0}H^0(S,\calO_S(nv'))$.
(This notation is taken from \cite{HK}.)
Thus $\mathbb{P}^1\cong Proj(R(S,\calO_S(f)))$ is the GIT quotient for the
linearization $v=s\in \chi(T_{S,G})_{\mathbb{Q}}$.

The proof for $G=A_n$ is similar.

\end{proof}

\section*{Appendix: two non-simple but semisimple cases}

Note that $G=E_{3}=A_{2}\times A_{1}$ and $G=D_{2}=A_{1}\times A_{1}$ are not
simple, but semisimple. For completeness, in these two cases, we define $G$-surfaces $(S,C)$ and the Cox rings $Cox(S,G)$
similarly as in Corollary~\ref{surface-description} and Definition~\ref{Cox-def}, and we compute briefly the
coordinate rings of the $G/P$ and the Cox rings explicitly. It turns out that there is
no embedding of $Spec(Cox(S,G))$ into $C(G/P)$.

\subsection*{(1) The case $G=E_{3}$}

Let $(S,C)$ be an $E_3$-surface, that is, $S$ is a blowup of a del Pezzo surface $X_{3}$ of degree $6$ at a general point, and $C$ is the exceptional curve.
Note that $X_{3}$ is a blowup of $\mathbb{P}^{2}$ at $3$ points in general position. The representation bundle
$\mathcal{L}_{E_3}$ is the tensor product of the standard representation bundles
of $A_{2}$ and $A_{1}$.

Precisely, $\mathcal{L}_{E_3}=\calV_{2}\otimes \calV_{1}$ where $\calV_{i} (i=1,2)$ is the standard
representation of $A_{i}$. By checking the highest weights, it is easy to see that

\begin{itemize}
\item[i)] $\mathcal{L}_{E_3}$ is determined by the set of $-1$ curves
$$\{l_{1},l_{2},l_{3},h-l_{1}-l_{2},h-l_{1}-l_{3},h-l_{2}-l_{3}\};$$

\item[ii)] $\calV_{2}^{*}$ is determined by $\{h-l_{1},h-l_{2},h-l_{3}\}$
and $\calV_{2}$ is determined by $\{-(h-l_{1}),-(h-l_{2}),-(h-l_{3})\}$;

\item[iii)] $\calV_{1}$ is determined by $\{h,2h-l_{1}-l_{2}-l_{3}\}$.
\end{itemize}

Note that $\calO_S(h)\bigoplus\calO_S(2h-l_{1}-l_{2}-l_{3})$ is a
standard representation bundle of the adjoint principal bundle $\mathscr{A}%
_{1}:=\calO_S\bigoplus\calO_S(\alpha_{1})\bigoplus\calO_S(-\alpha_{1})$
(recall that $\alpha_1=-h+l_1+l_2+l_3$).

The $Cox(S,E_{3})$ is defined as in Definition~\ref{Cox-def}. Then $Cox(S,E_{3})\cong Cox(X_{3})$,
and it is well-known that $Cox(X_{3})\cong \mathbb{C}[y_{1},\cdots,y_{6}]$, since
$X_{3}$ is toric (\cite{Cox}). Therefore
\[
Proj(Cox(S,E_3))\cong\mathbb{P}(H^{0}(S,\mathcal{L}_{E_3}))=\mathbb{P}^{5}.
\]
And $E_{3}/P=\mathbb{P}^{2}\times\mathbb{P}^{1}$. Denote $V=H^{0}(S,\mathcal{L}_{E_3})$.

Then we have
\[
Proj(Cox(S,E_3))=\mathbb{P}(V)\text{ (}\cong\mathbb{P}^{5}\text{),
}\mbox{ and }E_{3}/P=\mathbb{P}^{2}\times\mathbb{P}^{1}\hookrightarrow
\mathbb{P}(V),
\]
where the embedding $G/P\hookrightarrow\mathbb{P}(H^{0}(S,\mathcal{L}%
_{E_3}))=\mathbb{P}^{5}$ corresponds to the Segre embedding $\mathbb{P}%
^{2}\times\mathbb{P}^{1}\hookrightarrow\mathbb{P}^{5}$.

\subsection*{(2) The case $G=D_{2}$}

Let $(S,C)$ be a $D_{2}$-surface, that is, $S$ is a blowup of the ruled surface $\mathbb{P}^1\times\mathbb{P}^1$ or $\mathbb{F}_1$ at two points in general position, and $C$ is a smooth fiber.
The rank $4$ representation bundle
$\mathcal{L}_{D_{2}}$ is the tensor product of the standard representation
bundles of $A_{1}$ (recall that $D_{2}=A_{1}\times A_{1}$).

Note that $\mathcal{L}_{D_2}=\calO_S(l_{1})\bigoplus\calO_S%
(l_{2})\bigoplus\calO_S(f-l_{1})\bigoplus\calO_S(f-l_{2})$.
Let $\mathcal{V}_{1}:=\calO_S(l_{1}-s)\bigoplus\calO_S(l_{2}-s)$ and
$\mathcal{V}_{2}:=\calO_S(s)\bigoplus\calO_S(s+f-l_{1}-l_{2})$. Then we
find that $\mathcal{L}_{D_2}=\mathcal{V}_{1}\otimes\mathcal{V}_{2}$.
By checking the highest weights, we see that $\mathcal{V}_{1}, \mathcal{V}_{2}$
are the corresponding standard representations.

Thus we have $G/P\cong\mathbb{P}^{1}\times\mathbb{P}^{1}$. And the embedding
\[
G/P\hookrightarrow\mathbb{P}(H^{0}(S,\mathcal{L}_{D_2}))\cong\mathbb{P}^{3}%
\]
corresponds to the Segre embedding $\mathbb{P}^{1}\times\mathbb{P}%
^{1}\hookrightarrow\mathbb{P}^{3}$. On the other hand, the Cox ring
$Cox(S,D_{2})$, defined as in Definition~\ref{Cox-def}, is a sub-ring of $Cox(X_{3})$ generated by degree $1$
elements in $Cox(X_{3})$, since $S$ is also a del Pezzo surface $X_{3}$.
That $Cox(S,D_{2})$ is a sub-ring of $Cox(X_{3})$ follows directly from their definitions.
Therefore, we have
$Cox(S, D_2)=\mathbb{C}[x_{1},\cdots,x_{4}]$, and hence
\[
Proj(Cox(S,D_{2}))\cong\mathbb{P}(H^{0}(S,\mathcal{L}_{D_{2}}%
))\cong\mathbb{P}^{3}.
\]

\end{document}